\documentclass[letterpaper,12pt,openany,oneside]{article}

\usepackage[T1]{fontenc}
\usepackage[latin1]{inputenc}
\usepackage{epsfig}
\usepackage{amsmath,amssymb}
\usepackage[active]{srcltx}
\usepackage{dsfont}
\usepackage{mathtools}
\usepackage{xcolor,graphicx}
\usepackage{tikz}
\numberwithin{equation}{section}
\newtheorem{theoreme}{Theorem}[section]
\newtheorem{proposition}[theoreme]{Proposition}

\newtheorem{lemme}[theoreme]{Lemma}

\newenvironment{proof}[1][Proof]{\noindent \textbf{#1.}~ }
{\hfill\rule{2mm}{2mm} \vspace{\parskip} }

\newcommand{\R}{\ensuremath{\mathbb R}}

\newcommand{\E}{\ensuremath{\mathbb E}}
\providecommand{\prob}{\mathsf{P}}

\newcommand{\QQ}{\ensuremath{\mathbb Q}}

\newcommand{\N}{\ensuremath{\mathbb N}}

\newcommand{\val}{\mathsf{val}}
\newcommand{\bit}{\mathsf{bit}}

\newcommand{\ii}{\ensuremath{\mathbf i}}
\newcommand{\jj}{\ensuremath{\mathbf j}}

\newcommand{\Id}{\ensuremath{\operatorname{Id}}}

\newcommand{\De}{\Delta}

\newcommand{\la}{\lambda}
\newcommand{\ep}{\varepsilon}

\newcommand{\al}{\alpha}
\newcommand{\be}{\beta}
\newcommand{\si}{\sigma}
\newcommand{\ga}{\gamma}
\definecolor{darkblue}{rgb}{0,0,0.7} 

\title{New algorithms for solving stochastic games} 
\author{Miquel Oliu-Barton\footnote{Universit\'e Paris-Dauphine, PSL Research University, CNRS, CEREMADE, Paris, France.}}

\date{\today} 

\begin{document}
 \maketitle

\abstract{Stochastic games are a classical model in game theory in which two opponents interact and the environment changes in response to the players' behavior. The central solution concepts for these games are the discounted values and the value, which represent what playing the game is worth to the players for different levels of impatience.  In the present manuscript, we provide algorithms for computing exact expressions for the discounted values and for the value, which are polynomial in the number
of pure stationary strategies of the players. This result considerably improves all the existing algorithms, including the most efficient one, due to Hansen, Kouck\'y, Lauritzen, Miltersen and Tsigaridas (STOC 2011).} 

\setcounter{tocdepth}{2}
\tableofcontents

\section{Introduction}\label{intro}
\subsection{Motivation}
Concurrent stochastic games, henceforth stochastic games, were introduced by Shapley \cite{shapley53} in 1953 in order to model the dynamic interaction between two opponents. The theory of stochastic games and its applications have been studied in several scientific disciplines, including economics, operations research, evolutionary biology, and computer science. In addition, mathematical tools that were used and developed in the study of stochastic games are used by mathematicians and computer scientists in other fields. Stochastic games generalize matrix games and Markov decision problems; they 
are played over a finite set of states, and to each state corresponds a matrix game. 
Stochastic games are played in stages. 
At each stage $m\geq 1$, a stage reward $g_m$ is produced, which depends on the current state $k_m$, commonly observed by the players, and on the current pair of actions $(i_m,j_m)$ chosen by the players. 
The game is zero-sum, in the sense that Player~1 receives  $g_m$, while Player~2 receives $-g_m$. 
%
A $\la$-\emph{discounted stochastic game} is one where Player $1$ maximizes the expectation of the normalized $\la$-discounted sum $\sum_{m\geq 1}\la(1-\la)^{m-1} g_m$ for some discount rate $\la\in (0,1]$, while Player~2 minimizes the same amount. The case where the discount rate is close to $0$ is of particular importance, as it stands for the case where the players are patient. 
Alternatively, the interaction between patient players can be modeled by an \emph{undiscounted 
stochastic game}, that is: one in which Player $1$ maximizes the expectation of 
$\liminf_{T\to+\infty}\frac{1}{T}\sum\nolimits_{m=1}^T g_m$,
while Player~2 minimizes the same amount.

The central solution concept for zero-sum games is its {value}. When it exists, the value is the maximal amount that each player can obtain in expectation regardless of her opponent's behavior. 
The value of the $\la$-discounted stochastic game is often referred to as its $\la$-\emph{discounted value}, while the value of the undiscounted stochastic game is referred to as its \emph{value}. 
  In the present manuscript, we propose new algorithms for computing the discounted values and the value of any stochastic game, which are 
 %
 polynomial in the number of 
 pure stationary strategies of the game, that is: strategies that depend only on the current state. More precisely, for a stochastic game with $n$ states and $p$ actions available at each state, we provide explicit bounds which are polynomial in $p^n$. 
These results considerably improve 
all prior algorithms for computing the discounted values and the value of a stochastic game. In particular, they improve the best of them, due to Hansen, Kouck\'y, Lauritzen, Miltersen and Tsigaridas  \cite{HKMT11}, where the bounds are polynomial in $p$ and $2^{n^{n^2}}$. 

\paragraph{Notation.} In the sequel, we denote by $K=\{1,\dots,n\}$ the set of states, for some $n\in \N$. 
For any initial state $1\leq k\leq n$ any discount rate $\la\in(0,1]$, the $\la$-discounted value and the value of the stochastic game starting at $k$ are denoted, respectively, by $v_\la^k$ and $v^k$. We also set $v_\la:=(v^1_\la,\dots,v_\la^n)\in \R^n$ and $v:=(v^1,\dots,v^n)\in \R^n$.


\subsection{State of the art}\label{sofa}
In his seminal paper, Shapley \cite{shapley53} defined stochastic 
and proved that these games have a $\la$-discounted value for each $\la \in(0,1]$, and that both players have optimal stationary strategies, that is, strategies that depend only on the current state. Furthermore, 
a characterization was obtained for the vector of values $v_\la$, as the unique fixed point of an operator from $\R^n$ to $\R^n$ which is contracting for the $L^\infty$ norm. 
Blackwell and Ferguson \cite{BF68} considered a particular stochastic game, the so-called ``Big Match'', and proved that the existence of the value and the equality to $\lim_{\la\to 0}v_\la=v$. Their result was then extended by Kohlberg \cite{kohlberg74} to the class of absorbing games, that is, a class of stochastic games in which there is at most one transition between states. 
For general stochastic games, the convergence of the discounted values, as $\la$ goes to 0, was proved by Bewley and Kohlberg \cite{BK76}, building on Shapley's characterization of the discounted values and on Tarski-Seidenberg elimination theorem from semi-algebraic geometry. The existence of the value and the equality $v=\lim_{\la\to 0}v_\la$, were proved by Mertens and Neyman \cite{MN81}. 
An explicit characterization of the value was recently obtained by Attia and Oliu-Barton \cite{AOB18a}. 

Whether the value of a finite stochastic game can be computed in polynomial time is a famous open problem in computer science. This problem is intriguing because the simpler class of \emph{simple stochastic games} is both NP and co-NP,
and several famous problems with this property have  eventually been shown to be polynomial-time solvable, such
as primality testing or linear programming. (A simple stochastic game is one where, for each state, the transition function depends on one player's action only.) 
The known algorithms fall into two categories: 
decision procedures for the first order theory of the reals, such as 
\cite{chatterjee08, etessami2006, SV10}, and value or
strategy iteration methods, such as \cite{chatterjee2006,rao73}. All of them are worst-case exponential in the number of states or in the number of actions.  Hansen, Ibsen-Jensen and Miltersen  \cite{hansen2011} proved that no value or strategy iteration algorithm can ever achieve a polynomial bound. 
Recently, Hansen, Kouck\'y, Lauritzen, Miltersen and
Tsigaridas \cite{HKMT11} obtained a remarkable improvement using the machinery of real-algebraic geometry in a more indirect manner: they provided an algorithm which, for any fixed number of states, is polynomial in the number
of actions. However,
the dependence on the number of states is an implicit double exponential expression, which is problematic in terms of practical computations. In their own words (page 3 of  \cite{HKMT11}): ``\emph{the exponent in the polynomial time bound is $O(n)^{n^2}$, i.e., the complexity is doubly exponential in $n$}'', from which they claim that ``\emph{getting a better dependence on $n$ is a very interesting open problem}''.

\subsection{Main results}
In the present paper, we propose a new method for computing the $\la$-discounted value and the value of a stochastic game.  Unlike all prior works, we build on the new characterizations that were obtained by Attia and Oliu-Barton \cite{AOB18a}. 
Our algorithms are polynomial in the number of actions, for any \emph{fixed} number of states, but the dependence on the number of states is explicit and simply exponential. Equivalently, our algorithms are polynomial in the number of 
{pure stationary strategies}, that is, strategies that depend only on the current state. 
This improvement opens up the path for actually solving stochastic games in practice. An important ingredient in our work is the following \emph{continuity result}: for any $\ep>0$, we provide an explicit discount rate $\la_\ep\in(0,1]$ whose bit-size is 
is polynomial in the number of 
{pure stationary strategies} and in $\log \ep$, and so that 
$|v^k_\la-v^k|\leq \ep$ for all $\la\in(0,\la_\ep)$ and $1\leq k\leq n$. 

\subsection{Organisation of the paper} In Section~\ref{sec2} we provide a formal description of the model of stochastic games (Section~\ref{sec_mod}), we present our main results (Section~\ref{sec_main}), and we gather some results which are relevant for the sequel (Section~\ref{sec_res}). 
In Section~\ref{res1} we establish the above-mentioned continuity result together with some algebraic properties for the discounted values and the value of a stochastic game. 
In Section~4 we describe our new algorithms and establish Theorems 1 and 2. More precisely, Section~\ref{sec3} is devoted to the proof of Theorem~1, while Section~\ref{sec4} is devoted to the proof of Theorem~2. In both cases, we describe and analyze two algorithms, one which outputs arbitrarily close approximations of the desired value, namely $v^k_\la$ and $v^k$, respectively, and one which outputs these values exactly. 
%

\section{Stochastic games}\label{sec2}
We now introduce the model of stochastic games, and some basic facts. 
For a more detailed presentation of stochastic games, see for instance Sorin \cite[Chapter 5]{sorin02} and Renault \cite{renaultnotes2}. 
\subsection{Model and notation}\label{sec_mod}
We start by introducing some notation that will be used throughout the paper. \begin{itemize}
\item For each finite set $E$, we denote its cardinality by $|E|$ and the set of probability distributions over $E$ by $\De(E)=\{f:E\to [0,1],\ \sum_{e\in E}f(e)=1\}$.
\item We denote by $n$ the number of states.
\item $I^1,\dots,I^n$ and $J^1,\dots,J^n$ denote $2n$ fixed finite sets of actions.
\item We set $I:=I^1\times\dots\times I^n$ and $J:=J^1\times\dots\times J^n$.
\item We set $d:=\min(|I|,|J|)$. 
\item We set $X:=\De(I^1)\times \dots\times \De(I^n)$ and $Y:=\De(J^1)\times \dots\times \De(J^n)$.
\item For any $\al\in \R$, $\lceil \al \rceil$ denotes the unique integer satisfying $\al \leq \lceil \al \rceil< \al+1$.
\item For any $p\in \N$ we denote its bit-size by $\bit(p):=\lceil \log_2 (p+1)\rceil$.
\item For $(p,q)\in \N^2$, we set $\bit(p/q)=\bit(p)+\bit(q)$. 
\item For any tuplet of nonnegative integers $(a,b,c,M)$ we define 
\begin{eqnarray*} \varphi(a,b,c,M)&:=& 
 ab(\bit(a) + \bit(b)+\bit(c)+\bit(M)) \,. \end{eqnarray*}
In particular, for any $N\in \N$,  one has $\varphi(n,d,N,0)= nd (\bit(n) + \bit(d)+\bit(N))$. 
\end{itemize} 

\bigskip 
\noindent We can now describe the classical model of stochastic games, as in Shapley \cite{shapley53}.

\paragraph{Model.} 
A \emph{stochastic game} is described by a tuple $(K,I,J,g,q,k)$, where
\begin{itemize}
\item $K=\{1,\dots,n\}$ is a finite set of states. 
\item For each $1\leq \ell\leq n$, $I^\ell$ and $J^\ell$ are the sets of available actions for Player~1 and 2, respectively, at state $\ell$. 
\item $g:Z \to \R$ is the reward function, where $Z:=\{(\ell,i,j)\,|\, \ell\in K,\ (i,j)\in I^\ell\times J^\ell\}$. 
\item $q:Z\to \De(K)$ is the transition function. 
\item $1\leq k\leq n$ is the initial state.
\end{itemize}
The game proceeds in stages as follows. At each stage $m\geq 1$, both players are informed of the current state $k_m\in K$. Then, independently, Player~1 chooses an action $i_m\in I^{k_m}$ and Player~2 chooses an action $j_m\in J^{k_m}$. The pair $(i_m,j_m)$ is then observed by the players, from which they can infer the stage reward 
$g_m:=g(k_m,i_m,j_m)$. A 
 new state $k_{m+1}$ is then chosen  
with the probability distribution $q(k_m,i_m,j_m)$, and the game proceeds to stage $m+1$.\\

\noindent A \emph{$\la$-discounted stochastic game} is one where Player $1$ maximizes the expectation of  $\sum_{m\geq 1} \la(1-\la)^{m-1}g_m$ while Player $2$ minimizes the same amount, for some $\la\in(0,1]$.\\ 

\noindent An \emph{undiscounted stochastic game} is one where Player $1$ maximizes the expectation of  
$\liminf_{T\to +\infty} \frac{1}{T} \sum_{m=1}^T g_m$, while Player~2 minimizes the same amount.




\paragraph{Strategies.} 
A (behavioral) \emph{strategy} is a decision rule from the set of possible observations to the set of probabilities over the set of available actions. 
For every stage $m\geq 1$, the set of possible observations at stage $m$ is $Z^{m-1}\times K$. A strategy for Player~1 is thus a sequence of mappings $\si=(\si_m)_{m\geq 1}$ so that $\si_m(h_m,k_m)\in \De(I^{k_m})$ for all $(h_m,k_m)\in Z^{m-1}\times K$. 
%
Similarly, a strategy for Player~2 is a sequence of mappings $\tau=(\tau_m)_{m\geq 1}$ so that $\tau_m(h_m,k_m)\in \De(J^{k_m})$ for all $(h_m,k_m)\in Z^{m-1}\times K$. 
 Both players choose their strategies independently. The sets of strategies are denoted, respectively, by $\Sigma$ and $\mathcal{T}$. 
By the Kolmogorov extension theorem, the initial state $k$, the transition function $q$, and a pair of strategies $(\si,\tau)$ induce a unique probability over the set of plays $Z^{\N}$, endowed with the sigma-algebra generated by the 
cylinders corresponding to finite histories, i.e. the sets $(z_1,\dots,z_p)\times Z^\N$ for every $p\in \N$ and $(z_1,\dots,z_p)\in Z^p$. 
This probability is denoted by $\prob_{\si,\tau}^k$, and $\E_{\si,\tau}^k$ denotes the expectation with respect to $\prob_{\si,\tau}^k$.

\paragraph{Stationary strategies.} A \emph{stationary strategy} is one that depends on the past observations only through the current state. A stationary strategy of Player~1, denoted by $x=(x^1,\dots,x^n)$, is thus an element of $X$.
 Similarly, $y=(y^1,\dots,y^n)\in Y$ is a stationary strategy of Player~2. 
The sets $I$ and $J$ are the sets of \emph{pure stationary strategies}. We use 
the notation $\ii=(\ii^1,\dots,\ii^n)\in I$ and $\jj=(\jj^1,\dots,\jj^n)\in J$. 

\paragraph{Discounted and undiscounted payoffs.} 
To every pair $(\si,\tau)\in \Sigma\times \mathcal{T}$ corresponds a $\la$-discounted payoff for each discount rate $\la\in(0,1]$, and an undiscounted payoff, in the game $(K,I,J,g,q,k)$. They are given by 
\begin{eqnarray*}\label{payoff}\ga_\la^k(\si,\tau)&:=&\E_{\si,\tau}^k \left[\sum\nolimits_{m\geq 1}\la(1-\la)^{m-1} g_m\right], \\ 
\ga^k(\si,\tau)&:=&\E_{\si,\tau}^k \left[\liminf_{T\to+\infty}\frac{1}{T}\sum\nolimits_{m=1}^T g_m\right]\,. 
\end{eqnarray*}

\paragraph{The discounted values and the value.} For each discount rate $\la\in (0,1]$, the $\la$-discounted stochastic game  $(K,I,J,g,q,k)$ has a value, denoted by $v^k_\la$, whenever 
\[v^k_\la= \sup_{\si \in \Sigma} \inf_{\tau\in \mathcal T} \ga_\la^k(\si,\tau)=\inf_{\tau \in \mathcal T} \sup_{\si\in \Sigma} \ga_\la^k(\si,\tau)\,.\]
Similarly, the undiscounted stochastic game $(K,I,J,g,q,k)$ has a value, denoted by $v^k$,  whenever 
\[v^k=\sup_{\si \in \Sigma} \inf_{\tau\in \mathcal T} \ga^k(\si,\tau)=\inf_{\tau \in \mathcal T} \sup_{\si\in \Sigma} \ga^k(\si,\tau)\,.\]
\paragraph{Classical results.} The existence of $v_\la^k$ is due to Shapley \cite{shapley53}, while Mertens and Neyman \cite{MN81} proved the existence of $v^k$ and the equality $v^k=\lim_{\la\to 0} v^k_\la$. \\

In the sequel, we will refer to $v^k_\la$ and to $v^k$ as the $\la$-discounted value and the value, respectively, of the stochastic game $(K,I,J,g,q,k)$.







 
 \subsection{Main results}\label{sec_main} 
 In the sequel, we consider stochastic games which can be described with rational data. For any $N\in\N$, we say that the stochastic game $(K,I,J,g,q,k)$ satisfies $(H_N)$ if $g(\ell,i,j)$ and $q(\ell'\,|\,\ell,i,j)$ belong to the set $\{0,\frac{1}{N},\frac{2}{N},\dots,1\}$ for all $(\ell,i,j)\in Z$ and $1\leq \ell'\leq n$. Recall that $n=|K|$ is the number of states.\\
 
The main contributions of the present paper concern the computation of the discounted value $v^k_\la$ and the value $v^k$ of a stochastic game satisfying $(H_N)$. These numbers are known to be \emph{algebraic}, that is there exists polynomials $P$ and $Q$ with integer coefficients and so that $P(v^k_\la)=0$ and $Q(v^k)=0$. For an algebraic number $\al\in \R$, an \emph{exact expression for $\al$} is a triplet $(P; a,b)$ where $P$ is a polynomial with integer coefficients, 
$(a,b)$ is a pair of rational numbers, and $\al$ is the unique root of $P$ in the interval $(a,b)$. Thus, for instance, $(z^2-2; 1,2)$ is an exact expression for $\sqrt{2}$.

The complexity of the algorithms presented in this paper will be measured with the so-called \emph{logarithmic cost model} which consists in assigning, to every arithmetic operation, a cost that is proportional to the number of bits involved. An algorithm is \emph{polynomial in the variables $t_1,\dots,t_m$}, if its logarithmic cost can be bounded by a polynomial expression of $t_1,\dots,t_m$. \\

We can now state our results formally.  
\paragraph{Theorem~1.} \emph{There exists an algorithm that takes as input a stochastic game $(K,I,J,g,q,k)$ satisfying $(H_N)$ for some $N\in \N$ and 
a discount rate satisfying $\la\in \{0,\frac{1}{M},\frac{2}{M},\dots,1\}$ for some $M\in \N$, and outputs an exact expression for  its discounted value $v^k_\la$. The algorithm is polynomial in $n$, $|I|$, $|J|$, $\log N$ and $\log M$}. 

\paragraph{Theorem~2.} \emph{There exists an algorithm that takes as input a stochastic game $(K,I,J,g,q,k)$ satisfying $(H_N)$ for some $N\in \N$ and outputs an exact expression for its value $v^k$. The algorithm is polynomial in $n$, $|I|$, $|J|$ and $\log N$}.\\[0.2cm]

The algorithms that are mentioned in Theorems 1 and 2 are provided in Sections \ref{sec3} and \ref{sec4} respectively. Though very similar, the second algorithm has an additional ingredient, namely a new bound on 
how small the discount rate needs to be so that $v^k_\la$ and $v^k$ are close to each other. This result, which has an interest in its own, can be formalized as follows. 

\paragraph{Theorem~3.} \emph{For each $r\in \N$, set $\la_r:=2^{-4 nd(\bit(n)+\bit(d)+\bit(N))-rnd}$. Then, for any stochastic game $(K,I,J,g,q,k)$ satisfying $(H_N)$ one has  
$$\left| v_{\la}^k - v^k \right|\leq 2^{-r}\qquad \forall \la\in(0,\la_r]\,.$$}
\paragraph{Comments.} 
The previous results deserve some comments. For simplicity, we assume that, for some $p\in \N$, one has $|I^\ell|=|J^\ell|=p$  and for all $1\leq \ell\leq n$. Hence, in this case $|I|=|J|=d=p^n$.  
\begin{enumerate} 
\item Expressions that are polynomial in $|I|$ and $|J|$ are in fact exponential in $n$. In other words, the algorithms mentioned in Theorem~1 and 2 are not polynomial in the number of actions, $\sum_{\ell=1}^n |I^\ell|$ and $\sum_{\ell=1}^n |J^\ell|$ but rather in the number of pure stationary strategies, $\prod_{\ell=1}^n |I^\ell|$ and 
$\prod_{\ell=1}^n |J^\ell|$. Similarly, the  bit-size of $\la_r$ in Theorem~3 is {exponential} in $n$. 
\item Theorems 1 and 2 improve the algorithms provided by Hansen et al. \cite{HKMT11}. Our main achievement is two-fold: one the one hand, we reduce the dependence on $n$ from a double exponential to a simple exponential; on the other, our algorithms are considerably simpler and more direct. 
\item For the $\la_r$ in Theorem~3 one has $\bit(\la_r)=O(nd(r+ \bit(n)+\bit(d)+\bit(N))$, which is of order $p^n$. 
This result improves Proposition 22 of Hansen et al. \cite{HKMT11}, where an expression for the order of the bit-size of $\la_r$ is obtained in terms of big O's, namely, of order $p^{O(n^2)}$. 
Furthermore, the reduction from $p^{O(n^2)}$ to $p^n$ is fairly tight. Indeed, transposing Theorem 8 of \cite{hansen2011} into the discounted case, it follows that one can construct an example for which a discount rate $\la$ of bit-size $p^{n/2}$ is not enough to ensure that $v_\la^k$ and $v^k$ are close to each other. 
\end{enumerate}



\subsection{Selected past results} \label{sec_res}
We now gather some past results that will be used in our proofs. 
We start by defining the auxiliary matrices that were introduced in Attia and Oliu-Barton \cite{AOB18a}. Consider the play induced by a pair $(\ii,\jj)\in I\times J$ of {pure stationary strategies}. Every time that the state $1\leq \ell\leq n$ is reached, the players play $(\ii^\ell,\jj^\ell)\in I^\ell\times J^\ell$, so that the stage reward is $g(\ell,\ii^\ell,\jj^\ell)$ and the law of the next state is given by $q(\ell,\ii^\ell,\jj^\ell)$. 
Hence, the state variable follows a Markov chain with transition matrix $Q(\ii,\jj)\in \R^{n\times n}$ and the stage rewards can be described by a vector $g(\ii,\jj)\in \R^n$.
For any $\la\in(0,1]$, let
$\ga_\la(\ii,\jj):=(\ga^1_\la(\ii,\jj),\dots \ga^n_\la(\ii,\jj)) \in \R^n$ be the vector of expected payoffs in the $\la$-discounted game, as the initial state varies from $1$ to $n$.
By stationarity, $Q(\ii,\jj), g(\ii,\jj)$ and $\ga_\la(\ii,\jj)$ satisfy the recursive relation 
\begin{equation*}\label{qut}\ga_\la(\ii,\jj)=\la g(\ii,\jj) + (1-\la)Q(\ii,\jj) \ga_\la(\ii,\jj)\,.\end{equation*}
The matrix $\Id-(1-\la)Q(\ii,\jj)$ is invertible so that, 
by Cramer's rule, one has
\begin{equation}\label{qut2}\ga^k_\la(\ii,\jj)=\frac{d^k_\la(\ii,\jj)} {d^0_\la(\ii,\jj)},\end{equation}
where  $d^0_\la(\ii,\jj):=\det(\Id - (1-\la)Q(\ii,\jj))\neq 0$ 
and where $d^k_\la(\ii,\jj)$ is the determinant of the $n\times n$-matrix obtained by replacing the $k$-th column of $\Id - (1-\la)Q(\ii,\jj)$ with $\la g(\ii,\jj)$.  
\paragraph{The auxiliary matrix of \cite{AOB18a}.} The auxiliary matrix $W^k_\la(z)$ is obtained by linearizing the quotient in \eqref{qut2} with an auxiliary variable $z\in \R$. Formally, for any $z\in \R$, one defines  the $|I|\times |J|$ matrix $W^k_\la(z)$ by setting 
\begin{equation*}W^k_\la(z)[\ii,\jj]:=d^k_\la(\ii,\jj) - z d^0_\la(\ii,\jj),\quad \forall (\ii,\jj)\in I\times J\,.\end{equation*} 
\noindent Its value is denoted by $\val\, W^k_\la(z)$. 




\bigskip 

The following two results, which are the main object of \cite{AOB18a}, will be crucial in the sequel.
\begin{theoreme}\label{thm_AOB1} 
For any $\la\in(0,1]$, $v^k_\la$ is the unique $z\in \R$ so that $\val \, W_\la^k(z) =0$. Furthermore, the map $z\mapsto \val \, W_\la^k(z)$ is strictly decreasing. 
\end{theoreme}

\begin{theoreme}\label{thm_AOB2} 
$F^k(z):=\lim_{\la\to 0} \la^{-n}\,  \val\, W^k_\la(z)$ exists in $\R\cup\{\pm\infty\}$ for all $z\in \R$, and $v_\la^k$ converges, as $\la$ goes to 0, to the unique $w\in \R$ so that $z> w \Rightarrow  F^k(z)<0$ and $z< w  \Rightarrow  F^k(z)>0$. Furthermore, the map $z\mapsto F^k(z)$ is strictly decreasing. \end{theoreme}


The third result is contained in Theorem~2 of Shapley and Snow \cite{SS50}. 
For any matrix $M$ of size $p\times p$, we denote by 
$S(M)$ the sum of the entries of the adjugate matrix of $M$, with the convention $S(M)=1$ if $p=1$ (i.e. when the adjugate matrix is not defined). 
\begin{theoreme}\label{thm_ss} For any matrix $M$ of size $p\times q$, there exists a square sub-matrix of $M$, denoted by $\dot{M}$, so that $S(\dot{M})\neq 0$ and $\val\, M = \frac{\det \dot{M}}{S(\dot{M})}$. 
\end{theoreme} 

 \section{Algebraic properties of the values}\label{res1} 
Throughout this section, $(K,I,J,g,q,k)$ denotes a stochastic games satisfying $(H_N)$ for some $N\in \N$. 
Recall that a real number $\al$ is \emph{algebraic of degree $p$} if there exists a polynomial $P$ with integer coefficients satisfying $P(\al)=0$, and $p$ is the lowest degree of all such polynomials. The \emph{defining polynomial} of $\al$ is the unique polynomial with integer coefficients $P(z)=a_0+a_1z+\dots+a_p z^p$ so that $P(\al)=0$, $a_p> 0$ and $\gcd(a_0,\dots,a_p)=1$. In Section~\ref{boundcoeffs}, we combine a technical result from Basu, Pollack and Roy \cite{basu07} and Theorems~\ref{thm_AOB1}, \ref{thm_AOB2} and \ref{thm_ss} to establish new bounds for the degree and the coefficients of the defining polynomials of $v_\la^k$ and $v^k$. 
These results will be used to analyze the algorithms corresponding to Theorems 1 and 2. 
In Section~\ref{thm3}, we establish Theorem~3, a result that  reduces the computation of the value of a stochastic game to the computation of its discounted value, for a well-chosen discount rate. This result will be used in the algorithm corresponding to Theorem~2. 

 %
\subsection{Bounds on the defining polynomials of the values}\label{bound_coeffs} \label{boundcoeffs}
We start by recalling Proposition 8.12 of Basu, Pollack and Roy \cite{basu07}. 

\begin{lemme}\label{basu2} Let $A$ be an $p\times p$-matrix with polynomial entries in the variables $Y_1,\dots,Y_\ell$ of degrees bounded by $q$ and integer coefficients of bit-size at most $\nu$. Then $\det A$, considered as a polynomial in $Y_1,\dots,Y_\ell$ has degrees in $Y_1,\dots,Y_\ell$ bounded by $pq_1,\dots,pq_\ell$, and coefficients of bit-size at most $p\nu+p\bit(p)+\ell \bit(pq+1)$ where $q=\max(q_1,\dots,q_\ell)$.
\end{lemme}

\noindent We can now use Lemma~\ref{basu2} and Theorem~\ref{thm_ss} to prove the following result. 
\begin{lemme}\label{petit_lem} There exists two finite sets, denoted by $\mathcal{P}$ and $\mathcal{Q}$, which contain nonzero polynomials in the variables $(\la,z)$ of degree at most $nd$ in $\la$ and $d$ in $z$ and integer coefficients, so that for each $(\la,z)\in (0,1]\times \R$, there exists $P\in \mathcal{P}$ and $Q\in \mathcal{Q}$ so that $\val\, W^k_\la(z)=P(\la,z)/Q(\la,z)$, $Q(\la,z)\neq 0$. Moreover, the coefficients of $P$ are of bit-size at most $3\, \varphi(n,d,N,0)$.  
\end{lemme}
\begin{proof}Let $(\ii,\jj)\in I\times J$ be fixed. By construction, $W^k_\la(z)[\ii,\jj]=d^k_\la(\ii,\jj)-z d^0_\la(\ii,\jj)$, where $d^k_\la(\ii,\jj)$ and $d^0_\la(\ii,\jj)$ are the determinants of two $n\times n$ matrices whose entries are polynomial in $\la$ of degree at most one and with coefficients in the set $\{0,\frac{1}{N},\frac{2}{N},\dots,1\}$. Consequently, $N^n W^k_\la(z)$ is a polynomial in $\la$ and $z$, of degree at most $n$ and $1$ respectively, and integer coefficients whose bit-size is at most $\nu:=n \bit(n)+ n\bit(N)+\bit(n+1)$ by Lemma \ref{basu2}. 
Let $\mathcal{P}$ and $\mathcal{Q}$ be, respectively, the sets of nonzero polynomials obtained as $$P(\la,z):=\det(N^n \dot{W}^k_\la(z))\quad \text{  and }\quad Q(\la,z):=S(N^n \dot{W}^k_\la(z)),$$ when $\dot{W}^k_\la(z)$ ranges over all possible square sub-matrices of $W^k_\la(z)$. 
By Theorem~\ref{thm_ss}, there exists a pair $(P,Q)\in \mathcal{P}\times \mathcal{Q}$ so that $Q(\la,z)\neq 0$ and \[\val\, W^k_\la(z)=P(\la,z)/(N^n Q(\la,z)),\]
where the normalization of the denominator is due to the fact that, for any square matrix $M$ of size $p\in\N$ and any $\al\in \R$, one has $\det(\al M)=\al^p\det(M)$ while $S(\al M)=\al^{p-1}S(M)$. 
We now show that $P$ and $Q$ satisfy the desired properties. First, $P$ is nonzero as $z\mapsto \val\, W^k_\la(z)$ is strictly decreasing by Theorem~\ref{thm_AOB1}. Second, the sub-matrices of $W^k_\la(z)$ are of size at most $d$ so that, by Lemma~\ref{basu2}, all the polynomials in $\mathcal{P}$ and $\mathcal{Q}$ are of degree at most $n d$ in $\la$ and $d$ in $z$, and their coefficients are integers. From Lemma~\ref{basu2}, one also obtains a bound for the bit-size of the coefficients of $P$, namely $d \nu+d\bit(d)+2 \bit(nd+1)$. Replacing $\nu$ in the last expression we an expression that be easily bounded by $3\,\varphi(n,d,N,0)$, which gives the desired result. 
\end{proof}

\medskip
We are now ready to prove the main result of this section. Again, we assume that $\la$ is a multiple of $1/M$ for some $M\in \N$. 
\begin{proposition}\label{estim2} The defining polynomials of $v^k_\la$ and $v^k$ are of degree at most $d$ and have coefficients of bit-size at most $4\varphi(n,d,N,M)$ and $4\varphi(n,d,N,0)$, respectively. 
  \end{proposition}
\begin{proof} 
We start by proving the result for the discounted case. Let $\la$ be such that $\la M\in \N$. Let $P\in \mathcal{P}$ and $Q\in \mathcal{Q}$ be the two polynomials given in Lemma~\ref{petit_lem} for $z=v_\la^k$. Hence, $Q(\la,v_\la^k)\neq 0$ and $\val\, W^k_\la(v^k_\la)=P(\la,v_\la^k)/Q(\la,v_\la^k)$. By Theorem~\ref{thm_AOB1}, 
$\val\, W^k_\la(v^k_\la)=0$, and consequently $P(\la,v_\la^k)=0$ by choice of $P$. Now, as $\la M\in \N$ and $P(\la,z)$ is a nonzero polynomial of degree at most $nd$ in $\la$ and $d$ in $z$ with integer coefficients, the following expression 
$$P_\la(z):=M^{n d} P(\la,z)\qquad \forall z\in \R,$$
defines a nonzero polynomial of degree at most $d$ with integer coefficients and satisfying $P_\la(v_\la^k)=0$. Consequently, it is a multiple of the defining polynomial of $v^k_\la$. In particular, $v_\la^k$ has algebraic degree at most $d$. To bound the bit-size of the coefficients of $P_\la$ it is enough to use the bound $3\varphi(n,d,N,0)$ obtained in Lemma~\ref{petit_lem} for $P$, and to bound the bit-size of its factors we use the Landau-Mignotte bound (see Theorem 2 in \cite{mignotte74}), which adds an additional term $d+\log(d+1)$ to the previous bound. \\[0.2cm]
Consider now the undiscounted case. As already argued in the discounted case, for each $\la\in(0,1]$ there exists a nonzero polynomial $P\in \mathcal{P}$ of degree at most $nd $ in $\la$ and $d$ in $z$ (the choice of the polynomial depends on $\la$), with integer coefficients of bit-size at most $3\varphi(n,d,N,0)$, and so that $P(\la,v_\la^k)=0$. By finiteness of the set $\mathcal{P}$, and because two polynomials cannot intersect infinitely many times in $(0,1]$, one of these polynomials must satisfy $P(\la,v_\la^k)=0$ for all $\la$ sufficiently small. For this polynomial, denoted again by $P(\la,z)$, 
let $P_0,\dots,P_{nd}$ be the unique polynomials in $z$ so that 
\[P(\la,z)=P_0+\la P_1(z)+\dots+\la^{nd} P_{nd }(z)=0\,.\]
As $P(\la,z)$ is nonzero, there exists $0\leq s\leq nd$ and $P_s\neq 0$ so that 
\[P(\la,z)=\la^s P_s(z)+o(\la^{s}),\] 
By construction, $P_s$ is a nonzero polynomial of degree at most $d$ and has integer coefficients of bit-size at most $3\varphi(n,d,N,0)$.
Dividing by $\la^s$, and letting $\la$ go to $0$, 
\[0=\lim_{\la \to 0} \frac{P(\la,v^k_\la)}{\la^s}=P_s(v^k)\,.\]
Hence, $P_s$ is a multiple of the defining polynomial of $v^k$. Like in the discounted case, we obtain the desired bound from the Landau-Mignotte bound. 
\end{proof}

\paragraph{Comments.} This result, which relies on Theorems \ref{thm_AOB1}, \ref{thm_AOB2} and \ref{thm_ss}, improves the bound provided by Hansen et al. \cite{HKMT11}. To see this, consider the case where $|I^\ell|=|J^\ell|=p$ for some $p\in \N$ and all $1\leq \ell\leq n$ and $\la N\in \N$. In this case, \cite{HKMT11} bounded the algebraic degree of $v_\la^k$ and $v^k$ by $(2p+5)^n$, while Proposition~\ref{estim2} reduces the bound to $d=p^n$. Furthermore, this bound is tight. 
Our result also reduces the bound on the bit-size of the coefficients obtained therein, from $22p^2 n^2(2p+5)^n\bit(N)$ to $4np^n(\bit(n)+ \bit(p)+\bit(N))$. 


\subsection{The distance between $v^k_\la$ and $v^k$}\label{thm3}
In this section we establish Theorem~3. First of all, recall the following classical bounds from Cauchy \cite{cauchy} and Mahler \cite{mahler64} concerning the roots of polynomial. 
\begin{lemme}\label{sep1} Let $P(z)=a_0+a_1 z \dots+a_p z^\ell$ be a non zero polynomial with integer coefficients, and let 
$\|P\|_\infty:=\max(|a_0|,\dots,|a_p|)$ and $\|P\|_2:=(\sum_{r=0}^p a^2_r)^{1/2}$. Then, 
\begin{itemize}
 \item [$(i)$] If $\al$ is a root of $P$ then $\al\geq \frac{1}{2 \|P\|_\infty}$. 
\item [$(ii)$] If $\be\neq \al$ is another root of $P$ then 
$|\al-\be|\geq  p^{-(p+2)/2} \|P\|^{1-p}_2$. 
\end{itemize}
\end{lemme} 
The following result is a direct consequence of Proposition \ref{estim2} and Lemma \ref{sep1} $(ii)$. 
\begin{lemme} \label{sepsep} Let $P$ be the defining polynomial of $v_\la^k$,  and let $\ep\leq 2^{-8 d\varphi(n,d,N,M)}$. 
Then $P$ has no root in the interval $(v_\la^k-\ep, v_\la^k+\ep)$. Similarly, let $Q$ be the defining polynomial of $v^k$ and let 
$\ep\leq 2^{-8 d\varphi(n,d,N,0)}$. Then $Q$ has no root in the interval $(v^k-\ep, v^k+\ep)$.
\end{lemme}
\begin{proof} Let us start by $v^k_\la$. By definition of the defining polynomial $P(v^k_\la)=0$. By Proposition \ref{estim2}, $P$ is of degree at most $d$ and its integer coefficients are bounded by $C:=2^{4\varphi(n,d,N,M)}$. Consequently, $\|P\|^2_2 \leq C^2(d+1)$ and, by Lemma \ref{sep1} $(ii)$, any other root $z$ of $P$ satisfies 
\begin{eqnarray*} |z-v_\la^k|&\geq &d^{-(d+2)/2} (d+1)^{(1-d)/2} C^{1-d}\\ &\geq &
2^{-8 d \varphi(n,d,N,M)}\,. \end{eqnarray*}
This inequality proves the statement for $v_\la^k$. We omit the proof for $v^k$ as it goes along the exact same lines: it is enough to replace $P$, $v^k_\la$ and $\varphi(n,d,N,M)$ with $Q$, $v^k$ and $\varphi(n,d,N,0)$.
\end{proof}

\bigskip 
Using Lemma~\ref{sep1} $(i)$, Lemma~\ref{petit_lem}, and Theorem~\ref{thm_ss}, we now derive some valuable insight on the asymptotic behavior of the sign of the map $\la\mapsto \val\, W^k_\la(z)$ as $\la$ goes to $0$, for a well-chosen fixed $z\in \R$. This result will be crucial in the proof of Theorem~3. 

\begin{proposition}\label{sign} For any $r\in \N$, set $Z_r:=\{0, \frac{1}{2^r},\dots, \frac{2^r}{2^r}\}$ and 
 $\la_r:=2^{-4\varphi(n,d,N,0)-rnd}$. Then, for each $z\in Z_r$, 
$$\begin{cases} 
\val\, W^k_{\la_r}(z)>0 \ \Longrightarrow \ F^k(z)\in [0,+\infty]\\ 
\val\, W^k_{\la_r}(z)<0 \ \Longrightarrow \ F^k(z)\in [-\infty,0]\\ 
\val\, W^k_{\la_r}(z)=0 \ \Longrightarrow \ F^k(z)= 0\,.\end{cases}
$$
 \end{proposition} 
\begin{proof} Let $z\in Z_r$ be fixed. 
Let $\mathcal{P}$ and $\mathcal{Q}$ be the set of polynomials of Lemma~\ref{petit_lem}. 
Hence, for all $P\in\mathcal{P}$, the polynomial $P(\la,z)$ is of degree at most $nd$ in $\la$ and $d$ in $z$. Furthermore, by the choice of $z$, \[P_z(\la):=2^{r n d} P(\la,z)\qquad \la\in(0,1]\] 
defines a polynomial in the variable $\la$ of degree at most $nd$ and with integer coefficients of bit-size at most $3\varphi(n,d,N,0)+rnd+1$. Let $\mathcal{P}(z)$ and $\mathcal{Q}(z)$ be the set of all the polynomials obtained this way, as $P$ and $Q$ range, respectively, in the sets $\mathcal{P}$ and $\mathcal{Q}$. 
By Theorem~\ref{thm_ss}, for any $\la\in(0,1]$ there exists $P_z\in \mathcal{P}(z)$ and $Q_z\in \mathcal{Q}(z)$, the choice of the  polynomials depends on $\la$, so that
\[\val\, W^k_\la(z)=\frac{P_z(\la)}{Q_z(\la)}\,.\]
Hence, a necessary condition for the function $\la\mapsto \val\, W^k_\la(z)$ to change sign at some $\al\in \R$ is that $P_z(\al)=0$ for some polynomial $P_z\in \mathcal{P}(z)$. Applying Lemma~\ref{sep1} $(i)$ to 
the nonzero polynomials in $\mathcal{P}(z)$, it follows that neither of them admits a root in the interval 
$\left(0, \la_r\right]$. 
In other words, 
the sign of $\la\mapsto \val\, W^k_\la(z)$ is constant in the interval $(0,\la_r]$. Consider now the three possible cases, $\val\, W^k_{\la_r}(z)>0$, $\val\, W^k_{\la_r}(z)<0$, and $\val\, W^k_{\la_r}(z)=0$. In the first case, 
$\la^{-n}\, \val\,W^k_\la(z)>0$ for all $\la\in(0,\la_r]$ so that 
\[F^k(z):=\lim_{\la \to 0}\la^{-n}\, \val\, W^k_\la(z)\in [0,+\infty]\,.\]
The second case is similar. For the third, $\val\, W^k_{\la_r}(z)=0$ implies that $\val\, W^k_{\la}(z)=0$ for all $\la\in(0,\la_r]$  so that 
one also has $F^k(z)=0$. 
  \end{proof}
  
\bigskip 
We are now ready to establish Theorem~3, whose statement is as follows.\\ \emph{For each $r\in \N$, let $\la_r:=2^{-4\varphi(n,d,N,0)-rnd}$. Then $|v^k_\la-v^k|\leq 2^{-r}$ for all $\la\in(0,\la_r)$.}
\paragraph{Proof of Theorem~3.} 
 Let $\la\in(0,\la_r]$ be fixed. First of all, the maps $z\mapsto \val\, W_{\la}^k(z)$ and $z\mapsto F^k(z)$ are strictly decreasing, by Theorems \ref{thm_AOB1} and \ref{thm_AOB2},  
 respectively. 
Therefore, either there exists a unique $z\in Z_r=\{0, \frac{1}{2^r},\dots, \frac{2^r}{2^r}\}$ so that $\val\, W_{\la}^k(z)=0$, or 
there exists $0\leq m\leq 2^r$ such that $\val\, W_{\la}^k(m 2^{-r})>0$ and $\val\, W_{\la}^k((m+1) 2^{-r})<0$, and the same is true for $F^k$. Consider the first case, and let $z\in Z_r$ satisfy $\val\, W_{\la}^k(z)=0$. By Theorem~\ref{thm_AOB1}, this implies $z=v_{\la}^k$, so that, by Proposition~\ref{sign}, one also has 
 $F^k(z)=0$. But then, Theorem~\ref{thm_AOB2} implies $v^k=z$ so that $v_{\la}^k=v^k$, and the inequality $|v_\la^k-v^k|\leq 2^{-r}$ holds. Consider now the second case, and let $1\leq m\leq 2^r$ be such that  $\val\, W_{\la}^k(m 2^{-r})>0$ and $\val\, W_{\la}^k((m+1) 2^{-r})<0$. On the one hand, Theorem~\ref{thm_AOB1} implies 
\begin{equation}\label{eq1}m2^{-r}< v_{\la}^k < (m+1)2^{-r}\,.\end{equation} 
On the other, Proposition~\ref{sign} gives $F^k(m 2^{-r})\geq 0$ and $F^k((m+1) 2^{-r})\leq 0$ which, in view of Theorem~\ref{thm_AOB2}, implies 
\begin{equation}\label{eq2}
m2^{-r}\leq v^k \leq (m+1)2^{-r}\,.\end{equation} 
The combination of \eqref{eq1} and \eqref{eq2} yields the desired inequality $|v_\la^k-v^k|\leq 2^{-r}$. \hfill $\square$

\section{Algorithms}\label{sec_algo}
The aim of this section is to describe the algorithms that correspond to Theorems 1 and 2. 
We start by recalling three classical algorithms in Section~\ref{aux} that are called by the above-mentioned algorithms. Section~\ref{sec3} is devoted to the description of two algorithms: the first one outputs arbitrarily close approximations for the discounted values of a stochastic game, while the second one outputs an exact expression for this value. The latter corresponds to the algorithm of Theorem~1. Similarly, Section~\ref{sec3} is devoted to the description of three algorithms: the first two output arbitrarily close approximations for the value of a stochastic game, while the third one outputs anexact expression for this value.  The latter corresponds to the algorithm of Theorem~2. 

\subsection{Auxiliary results} \label{aux}
Recall that the complexity of the algorithms is measured with the logarithmic cost model. The logarithmic cost of an algorithm can be bounded by 1) a bound of the number of arithmetic operations that it requires, and 2) a bound the bit-size of the numbers that are involved in them. In particular, if these two bounds are polynomial expressions in some variables $t_1,\dots,t_m$, so is the logarithmic cost of the algorithm. 

We now recall three well-known algorithms. 
The first one, due to Kannan, Lenstra and Lov\'asz \cite{KLL88}, allows to compute the defining polynomial of an algebraic number efficiently. It will be referred as the \emph{KLL algorithm}. The second, due to Karmarkar \cite{karmarkar84}, allows to solve linear programs efficiently and will be referred to as the \emph{Karmarkar algorithm}. The third one, which allows to compute the determinant of a square matrix efficiently, is taken from Basu, Pollack and Roy \cite{basu07}, where it is referred to as the \emph{Dodgson-Jordan-Bareiss algorithm}.  

\begin{theoreme}\label{KLL} 
Let $\al\in \R$ be an algebraic number of degree $p\in \N$ and defining polynomial $P(z)=a_0+a_1z+\dots+a_p z^p$, $z\in \R$. 
 The \emph{KLL algorithm} outputs the defining polynomial of $\al\in \R$ when given as inputs $(D, C, \bar{\al})$ satisfying
\begin{itemize}
\item $B\in \N$ is a bound on the algebraic degree of $\al$, i.e. $p\leq D$.
\item $2^C\in \N$ bounds the integer coefficients $|a_0|,\dots, |a_p|$. 
\item $\bar{\al}\in \QQ$ so that $|\al-\bar{\al}|\leq 2^{-s}\frac{1}{12D}$, where 
$$s=s(b,C):=\lceil D^2/2+(3D+4)\log_2(D+1)+2 D C\rceil\,.$$
\end{itemize}
This algorithm 
 requires $O(p D^4(D+C))$ arithmetic operations on integers of bit-size $O(D^2(D+ C))$.
 \end{theoreme} 

\begin{theoreme}\label{Karma} Let $M$ be a $p\times q$ matrix with rational entries which can be encoded in $C$ bits. The \emph{Karmarkar algorithm} 
inputs $M$ and outputs its value $\val\, M$, and requires $O(p^{3.5} C)$ arithmetic operations on integers of bit-size at most $O(C)$. 
\end{theoreme}  

\begin{theoreme}\label{DJB} Let $M$ be a $p\times p$ matrix with integer entries of bit-size $C$. 
 The \emph{Dodgson-Jordan-Bareiss algorithm} inputs $M$ and outputs $\det \, M$, and requires $O(p^3)$ arithmetic operations on integers of bit-size $O(p\, \bit(p) C)$. \end{theoreme}

 For the three above-mentioned algorithms, the following assertions hold. 
\begin{itemize}
\item The \emph{KLL algorithm} is polynomial in $D$ and $C$, the bounds for the degree and the bit-size of the coefficients, respectively, of the defining polynomial of $\al$. 
\item The \emph{Karmarkar algorithm} and the \emph{Dogson-Jordan-Bareiss algorithm} are polynomial in the size of the matrix and in the bit-size of its entries. 
\end{itemize}


\subsection{Computing the discounted values}\label{sec3}
We start by describing a bisection algorithm, directly derived from Theorem~\ref{thm_AOB1}, that outputs arbitrarily close approximations of the $\la$-discounted value $v_\la^k$ of a stochastic game $(K,I,J,g,q,k)$. As we will show later on (see \emph{Algorithm 2 approx bis}), this algorithm can also be used to obtain arbitrarily close approximations of $v^k$, thanks to Theorem~3. 

\bigskip
\hrule  
\smallskip
\noindent \textbf{\emph{Algorithm 1 approx}}
\smallskip
\hrule 
\bigskip

\noindent \textbf{Input:} A stochastic game $(K,I,J,g,q,k)$ satisfying $(H_N)$ for some $N\in \N$, a discount rate $\la$ satisfying $\la M \in \N$ for some $M\in \N$, and a precision level $r\in \N$.\\[0.3cm]%
\textbf{Output:} A rational number $u\in \{0,\frac{1}{2^r},\dots,\frac{2^{r}-1}{2^r}\}$ so that $v^k_\la\in [u, u+\frac{1}{2^{r}}]$.\\[0.3cm] 
\textbf{Computation cost:} Polynomial in $n$, $|I|$, $|J|$, $\log N$, $\log M$ and $r$. \\[0.1cm] 

\noindent $1$. Set $\underline{w}:=0$, $\overline{w}:=1$\\[0.15cm] 
\noindent $2.$ WHILE $\overline{w}-\underline{w}>2^{-r}$ DO  
\begin{enumerate}
\item[$2.1$] $z:= \frac{\underline{w}+\overline{w}}{2}$
\item[$2.2$] Compute $W^k_{\la}(z)$ 
\item[$2.3$] Compute $v:= \mathrm{val}\, W^k_{\la}(z)$
\item[$2.4$] IF $v\geq 0$, THEN $\underline{w}:=z$ 
\item[$2.5$] IF $v \leq 0$ THEN $\overline{w}:=z$
\end{enumerate}
\noindent $3.$ RETURN $u:= \underline{w}$\,.  \\[0.15cm] \hrule 
\bigskip


\paragraph{Computation cost of \emph{Algorithm 1 approx}.} 
By Theorem~\ref{thm_AOB1}, each iteration of Step 2 reduces the interval $[\underline{w},\overline{w}]$ by a factor of $1/2$, while satisfying $\underline{w}\leq v^k_\la\leq \overline{w}$. Consequently, the algorithm terminates after at most $r$ steps, and the output $u$ satisfies $v^k_\la\in [u, u+\frac{1}{2^r}]$. As Steps 2.1, 2.4 and 2.5 require one operation each, the computation cost of the algorithm depends essentially on the computation cost of Steps 2.2 and 2.3, which is the object of the following lemma.

\begin{lemme}\label{estim} Let $L=O(n \bit(n)\bit(N)\bit(M))$. For all $r\in \N$ and $z\in\{0,\frac{1}{2^r},\dots,\frac{2^{r}-1}{2^r}\}$, 
\begin{itemize} 
\item[$(i)$]  
The computation of $W^k_\la(z)$ with the \emph{Dogson-Jordan-Bareiss algorithm}   
requires
$O(n^3|I||J|)$ arithmetic operations, with numbers of bit-size $O(L+r)$.
\item[$(ii)$]  The computation of $\val\, W^k_\la(z)$ with the \emph{Karmarkar algorithm} requires   
$O(d^{3.5}(L+r))$ arithmetic operations, with numbers of bit-size $O(L+r)$.
\end{itemize}
\end{lemme}
\begin{proof}
Let $r\in \N$ and $z\in \{0,\frac{1}{2^r},\dots,\frac{2^{r}-1}{2^r}\}$ be fixed. \\[0.1cm]
$(i)$ Thanks to the assumptions $(H_N)$ and $\la M\in \N$, and by the definition of $W^k_\la(z)$, for each $(\ii,\jj)\in I\times J$, $W^k_\la(z)[\ii,\jj]=d_\la^0(\ii,\jj)-z d_\la^0(\ii,\jj)$ where $d_\la^0(\ii,\jj)$ and $d_\la^k(\ii,\jj)$ are determinants of some $n\times n$ matrices whise entries are multiples of $\frac{1}{NM}$. 
Multiplying each entry by $NM$, so that all entries are integers, it follows then from Theorem~\ref{DJB} that 
$d_\la^0(\ii,\jj)$ and $d_\la^0(\ii,\jj)$ can be computed in $O(n^3)$ arithmetic operations on integers of bit-size 
$O(n\bit(n) \bit(NM))= 
O(L)$. The entries of $W^k_\la(z)$ are then of bit-size at most $O(L+r)$ because, by the choice of $z$, $\bit(z)=O(r)$. 
Finally, the total number of operations is simply $O(n^3 |I| |J|)$ because the matrix $W^k_\la(z)$ is of size $|I|\times |J|$. \\[0.1cm] 
$(ii)$ As already noted in the poof of $(i)$, the entries $W^k_\la(z)$ are of bit-size at most $O(L+r)$. The result follows then directly from Theorem~\ref{Karma}. 
%
\end{proof}

\bigskip 
The next result is a direct consequence from Lemma~\ref{estim}. 

\begin{theoreme}\label{approx} 
\emph{Algorithm 1 approx} computes a $2^{-r}$-approximation of $v^k_\la$ for any $r\in \N$, and its computation cost is polynomial in $n$, $|I|$, $|J|$, $\log N$, $\log M$ and $r$. 
\end{theoreme} 


Next, we combine \emph{Algorithm 1 approx} and the \emph{KLL algorithm} in order to obtain an exact expression for $v^k_\la$. 


\bigskip 

\hrule  
\smallskip
\noindent \textbf{\emph{Algorithm 1 exact}}
\smallskip
\hrule 
\bigskip

\noindent \textbf{Input:} A stochastic game $(K,I,J,g,q,k)$ satisfying $(H_N)$ for some $N\in \N$, and a discount rate $\la\in(0,1]$ so that $\la M\in\N$ for some $M\in \N$. \\[0.3cm]%
\textbf{Output:} An exact expression for $v_\la^k$. \\[0.3cm] 
\textbf{Computation cost:} Polynomial in $n$, $|I|$, $|J|$, $\log N$ and $\log M$.\\[0.1cm] 

\noindent $1$. Initialization phase
\begin{enumerate}
\item[$1.1$]Set $C:=4\varphi(n,d,N,M)$ 
\item[$1.2$] Set $s:=\lceil d^2/2+(3d+4)\log_2(d+1)+2 d  C\rceil$
\item[$1.3$] Set $r:= s \lceil \log_2 12 d \rceil $
\end{enumerate}
\noindent $2$. 
Run \emph{Algorithm 1 approx} with inputs $(K,I,J,g,q,k)$, the discount rate $\la$, and a precision level $r$. Denote its output by $u$.\\[0.25cm] 
\noindent $3$. Run the \emph{KLL algorithm} with inputs $d$, $C$ and $u$, and output $P$\,.\\[0.25cm]
\noindent $4$. RETURN $(P; u, u+2^{-r})$\,. 
 \bigskip 
 \hrule 
\bigskip
\bigskip

We are now ready to prove Theorem~1. That is, that \emph{Algorithm 1 exact} computes an exact expression for $v^k_\la$, and that its computation cost is polynomial in $n$, $|I|$, $|J|$, $\log N$ and $\log M$. 

\paragraph{Proof of Theorem~1.} 
First, recall that the algebraic degree of $v^k_\la$, and the bit-size of the coefficients of its defining polynomial, are bounded by $d$ and $C$ respectively, by Proposition~\ref{estim2}. 
Second, by Theorem~\ref{approx}, Step 2 of \emph{Algorithm 1 exact} returns $u$ so that $\bit(u)\leq 2r$ and 
$|u -v_{\la}^k|\leq 2^{-r}$, and the computation cost is polynomial in $n$, $|I|$, $|J|$, $\log N$, $\log M$ and $r$. 
Third, by  Theorem~\ref{KLL}, the definition of $C$, $r$ and $s$ in Step 1 of \emph{Algorithm 1 exact} ensure that Step 3 of \emph{Algorithm 1 exact} provides the defining polynomial $P$ of $v_\la^k$, and that the computation cost is polynomial in $d$ and $C$. As $C$ and $r$ are (bounded by) polynomial expressions in $n$, $d$ and $\log N$, the entire algorithm is thus polynomial in $n$, $|I|$, $|J|$, $\log N$ and $\log M$. It remains to show that $P$ has no other root than $v^k_\la$ in the interval $(u, u+2^{-r})$ so that $(P; u, u+2^{-r})$ is an exact expression for $v^k_\la$. 
To see this, note that by definition one has $r\geq 8 d\varphi(n,d,N,M)$. By Lemma \ref{sepsep}, this implies that $P$ has no other root in the interval $(v_\la^k-2^{-r}, v_\la^k+2^{-r})$, and the result follows because this interval contains $(u, u+2^{-r})$ thanks to $|u -v_{\la}^k|\leq 2^{-r}$.  \hfill $\square$ 

\subsection{Computing the value}\label{sec4}
Like for the discounted case, we start by proposing a bisection algorithm, directly derived from Theorem~\ref{thm_AOB2}, which outputs arbitrarily close approximations of the value $v^k$ of a stochastic game $(K,I,J,g,q,k)$. Note, however, that the natural algorithm would consist in computing the sign of $F^k(z):=\lim_{\la\to 0} \la^{-n}\val\, W^k_\la(z)$ at each iteration, but this computation seems very costly. 
Luckily, there is a way out to this issue. Indeed, by Proposition~\ref{sign}, this computation is equivalent to that of the sign of $\val\, W^k_\la(z)$, for a well-chosen $\la$, and this can be done efficiently because it is a linear program (provided that the bit-size of $\la$ is polynomial). The following bisection algorithm is built upon this observation. 

\bigskip
\hrule  
\smallskip
\noindent \textbf{\emph{Algorithm 2 approx}}
\smallskip
\hrule 
\bigskip

\noindent \textbf{Input:} A stochastic game $(K,I,J,g,q,k)$ satisfying $(H_N)$ for some $N\in \N$, and a precision level $r\in \N$.\\[0.3cm]%
\textbf{Output:} A rational number $u\in \{0,\frac{1}{2^r},\dots,\frac{2^r-1}{2^r}\}$ so that $v^k\in [u, u+\frac{1}{2^{r}}]$.\\[0.3cm] 
\textbf{Computation cost:} Polynomial in $n$, $|I|$, $|J|$, $\log N$ and $r$.\\[0.1cm] 

\noindent $1.1$. Set $\la_r:=2^{-4\varphi(n,d,N,0)-rnd}$ \\[0.15cm] 
\noindent $1.2$. Set $\underline{w}:=0$, $\overline{w}:=1$\\[0.15cm] 
\noindent $2.$ WHILE $\overline{w}-\underline{w}>2^{-r}$ DO 
\begin{enumerate}
\item[$2.1$] $z:= \frac{\underline{w}+\overline{w}}{2}$
\item[$2.2$] Compute $W^k_{\la_r}(z)$ 
\item[$2.3$] Compute $v:=\val\, W^k_{\la_r}(z)$
\item[$2.4$] IF $v\geq 0$, THEN $\underline{w}:=z$ 
\item[$2.5$] IF $v \leq 0$ THEN $\overline{w}:=z$
\end{enumerate}
\noindent $3.$ RETURN $u:= \underline{w}$\,. \\[0.15cm] \hrule 
\bigskip
\bigskip

The next result is a direct consequence of Proposition~\ref{sign}, Lemma~\ref{estim} and the definition of $\la_r$. 

\begin{theoreme}\label{approx2} \emph{Algorithm 2 approx} computes a $2^{-r}$-approximation of $v^k$ for any $r\in \N$, and its computation cost is polynomial in $n$, $|I|$, $|J|$, $\log N$ and $r$. 
\end{theoreme} 
\begin{proof} By Proposition~\ref{sign}, the sign of $W^k_{\la_r}(z)$ coincides with the sign of $F^k(z)=\lim_{\la\to 0} \la^{-n}\val\, W^k_\la(z)$ at every $z$ that is called by the algorithm. It follows then from Theorem~\ref{thm_AOB2} that \emph{Algorithm 2 approx} provides a $2^{-r}$-approximation of $v^k$. 
By Lemma~\ref{estim}, its computation cost is polynomial in $n$, $d$, $\bit(N)$ and $\bit(\la_r)$. The result follows then from the fact that the bit-size of $\la_r$ is polynomial in $n$, $d$, $\log N$ and $r$.
\end{proof}

\bigskip 

Alternatively, one can use Theorem~3 to obtain arbitrary close approximation for $v^k$ directly from \emph{Algorithm 1 approx}, as follows. 

\bigskip
\hrule  
\smallskip
\noindent \textbf{\emph{Algorithm 2 approx bis}}
\smallskip
\hrule 
\bigskip

\noindent \textbf{Input:} A stochastic game $(K,I,J,g,q,k)$ satisfying $(H_N)$ for some $N\in \N$, and a precision level $r\in \N$. \\[0.3cm]%
\textbf{Output:} A rational number $u\in \{0,\frac{1}{2^r},\dots,\frac{2^r-1}{2^r}\}$ so that $v^k\in [u, u+\frac{1}{2^{r}}]$.\\[0.3cm] 
\textbf{Computation cost:} Polynomial in $n$, $|I|$, $|J|$, $\log N$ and $r$.\\[0.1cm] 

\noindent $1$. Set $\la_{r+1}:=2^{-4\varphi(n,d,N,0)-(r+1)nd}$. \\[0.15cm] 
\noindent $2$. Run \emph{Algorithm 1 approx} with inputs the stochastic game $(K,I,J,g,q,k)$, the discount rate $\la_{r+1}$ and the precision level $r+1$. Let $u$ denote its output.\\[0.15cm] 
\noindent $3.$ RETURN $u$\,. \\[0.15cm] \hrule 
\bigskip


\begin{theoreme}\label{approx2} \emph{Algorithm 2 approx bis} computes a $2^{-r}$-approximation of $v^k$ for any $r\in \N$, and its computation cost is polynomial in $n$, $|I|$, $|J|$, $\log N$ and $r$. 
\end{theoreme} 
  \begin{proof} By Theorem~\ref{approx}, Step 2 of \emph{Algorithm 2 approx bis} outputs $u$ so that $|u-v^k_{\la_{r+1}}|\leq \frac{1}{2^{r+1}}$, and the cost is polynomial in $n$, $|I|$, $|J|$, $\log N$, $r$ and $\bit(\la_r)$. The latter being polynomial $n$, $|I|$, $|J|$, $\log N$ and $r$, the cost is thus polynomial in these variables too. Finally, by the choice of $\la_{r+1}$, Theorem~3 implies
  $|v^k-v^k_{\la_{r+1}}|\leq \frac{1}{2^{r+1}}$. The result follows, since 
  $$|u-v^k|\leq |u-v^k_{\la_{r+1}}|+ |v^k-v^k_{\la_{r+1}}|\leq \frac{1}{2^{r}}\,.$$ 
\end{proof}
 
 \bigskip  
Like in the discounted case, one can now combine \emph{Algorithm 2 approx} (or \emph{Algorithm 2 approx bis}) with the \emph{KLL algorithm} to obtain an algorithm that outputs an exact expression for $v^k$. 
The algorithm goes as follows. 

\bigskip 

\hrule  
\smallskip
\noindent \textbf{\emph{Algorithm 2 exact}}
\smallskip
\hrule 
\bigskip

\noindent \textbf{Input:} A finite stochastic game $(K,I,J,g,q,k)$ satisfying $(H_N)$ for some $N\in \N$. \\[0.3cm]%
\noindent \textbf{Output:} An exact expression for $v^k$.\\[0.3cm]  
\textbf{Computation cost:} Polynomial in $n$, $|I|$, $|J|$ and $\log N$.\\[0.1cm]

\noindent $1$. Initialization phase.
\begin{enumerate}
\item[$1.1$]Set $C:=4\varphi(n,d,N,0)$. 
\item[$1.2$] Set $s:=\lceil d^2/2+(3d+4)\log_2(d+1)+2 d C\rceil$.
\item[$1.3$] Set $r:= s \lceil \log_2 12 d \rceil $.
\end{enumerate}
\noindent $2$. Run \emph{Algorithm 2 approx bis} with inputs the stochastic game $(K,I,J,g,q,k)$ and the precision level $r+1$. Denote its output by $u$. \\[0.25cm] 
\noindent $3$. Run the \emph{KKL algorithm} with inputs $d$, $C$ and $u$. Denote its output by $Q$. \\[0.25cm]
\noindent $4$. RETURN $(Q; u, u+2^{-r})$\,. 
  \bigskip 
 \hrule 
\bigskip
\bigskip

We are now ready to prove Theorem~2, that is: \emph{Algorithm 2 exact} computes an exact expression for $v^k$, and its computation cost is polynomial in $n$, $|I|$, $|J|$ and $\log N$. The proof is similar to that of 
Theorem~1.

\paragraph{Proof of Theorem~2.} 
 First, recall that the algebraic degree of $v^k$, and the bit-size of the coefficients of its defining polynomial, are bounded by $d$ and $C$ respectively, by Proposition~\ref{estim2}. 
Second, by Theorem~\ref{approx2}, Step 2 of \emph{Algorithm 2 exact} returns $u$ so that $\bit(u)\leq 2r$ and 
$|u -v^k|\leq 2^{-r}$, and the computation cost is polynomial in $n$, $|I|$, $|J|$, $\log N$, and $r$. 
Third, by  Theorem~\ref{KLL}, the definition of $C$, $r$ and $s$ in Step 1 of \emph{Algorithm 1 exact} ensure that Step 3 of \emph{Algorithm 2 exact} provides the defining polynomial $Q$ of $v^k$, and that the computation cost is polynomial in $d$ and $C$. As $C$ and $r$ are (bounded by) polynomial expressions in $n$, $d$ and $\log N$, the entire algorithm is thus polynomial in $n$, $|I|$, $|J|$ and $\log N$. It remains to show that $Q$ has no other root than $v^k$ in the interval $(u, u+2^{-r})$ so that $(Q; u, u+2^{-r})$ is an exact expression for $v^k$. 
To see this, note that by definition one has $r\geq 8 d\varphi(n,d,N,0)$. By Lemma \ref{sepsep}, this implies that $Q$ has no other root in the interval $(v^k-2^{-r}, v^k+2^{-r})$, and the result follows because this interval contains $(u, u+2^{-r})$ thanks to $|u -v^k|\leq 2^{-r}$. \hfill $\square$ 

\section*{Acknowledgements} 
I am very much indebted to Krishnendu Chatterjee for his useful comments and time, and to Kristoffer Hansen for his insight and advice. I am also thankful to the comments of the anonymous referees of the journal, which have greatly contributed in the presentation and organization of the results. 
Finally, I gratefully acknowledge the support of the French National Research Agency, under
grant ANR CIGNE (ANR-15-CE38-0007-01), and the support of the Cowles Foundation at Yale University.

\bibliographystyle{amsplain}
\bibliography{bibliothese3}

\providecommand{\bysame}{\leavevmode\hbox to3em{\hrulefill}\thinspace}
\providecommand{\MR}{\relax\ifhmode\unskip\space\fi MR }
\providecommand{\MRhref}[2]{%
  \href{http://www.ams.org/mathscinet-getitem?mr=#1}{#2}
}
\providecommand{\href}[2]{#2}
\begin{thebibliography}{10}

\bibitem{AOB18a}
L.~{Attia} and M.~{Oliu-Barton}, \emph{A formula for the value of a stochastic
  game}, ArXiv: 1809.06102. Proceedings of the National Academy of Sciences of
  the United States of America, 2020.

\bibitem{basu07}
S.~Basu, R.~Pollack, and M.-F. Roy, \emph{Algorithms in real algebraic
  geometry}, vol.~10, Springer Science \& Business Media, 2007.

\bibitem{BK76}
T.~Bewley and E.~Kohlberg, \emph{The asymptotic theory of stochastic games},
  Mathematics of Operations Research \textbf{1} (1976), 197--208.

\bibitem{BF68}
D.~Blackwell and T.S. Ferguson, \emph{The {B}ig {M}atch}, Annals of
  Mathematical Statistics \textbf{39} (1968), 159--163.

\bibitem{cauchy}
A.~Cauchy, \emph{Calcul des indices des fonctions}, Journal de l'\'Ecole
  Polytechnique \textbf{15} (1832), no.~25, 176--229.

\bibitem{chatterjee08}
K.~Chatterjee, R.~Majumdar, and T.A. Henzinger, \emph{Stochastic limit-average
  games are in {EXPTIME}}, International Journal of Game Theory \textbf{37}
  (2008), 219--234.

\bibitem{etessami2006}
K.~Etessami and M.~Yannakakis, \emph{Recursive concurrent stochastic games},
  International Colloquium on Automata, Languages and Programming. {P}art {II},
  Lecture Notes in Computer Science, vol. 4052, Springer, Berlin, 2006,
  pp.~324--335.

\bibitem{hansen2011}
K.A. Hansen, R.~Ibsen-Jensen, and P.B. Miltersen, \emph{The complexity of
  solving reachability games using value and strategy iteration}, International
  Computer Science Symposium in Russia, Springer, 2011, pp.~77--90.

\bibitem{HKMT11}
K.A. Hansen, M.~Kouck\'y, N.~Lauritzen, P.B. Miltersen, and E.P. Tsigaridas,
  \emph{Exact algorithms for solving stochastic games}, Proc. 43rd Annual ACM
  Symposium on Theory of Computing, 2011, pp.~205--214.

\bibitem{chatterjee2006}
T.A. Henzinger, L.~de~Alfaro, and K.~Chatterjee, \emph{Strategy improvement for
  concurrent reachability games}, Third International Conference on the
  Quantitative Evaluation of Systems (QEST), IEEE, 2006, pp.~291--300.

\bibitem{KLL88}
R.~Kannan, A.~K. Lenstra, and L.~Lov{\'a}sz, \emph{Polynomial factorization and
  nonrandomness of bits of algebraic and some transcendental numbers},
  Mathematics of Computation \textbf{50} (1988), 235--250.

\bibitem{karmarkar84}
N.~Karmarkar, \emph{A new polynomial-time algorithm for linear programming},
  Proc. 16th Annual ACM Symposium on Theory of Computing, ACM, 1984,
  pp.~302--311.

\bibitem{kohlberg74}
E.~Kohlberg, \emph{Repeated games with absorbing states}, Annals of Statistics
  \textbf{2} (1974), 724--738.

\bibitem{mahler64}
K.~Mahler, \emph{An inequality for the discriminant of a polynomial}, The
  Michigan Mathematical Journal \textbf{11} (1964), no.~3, 257--262.

\bibitem{MN81}
J.-F. Mertens and A.~Neyman, \emph{Stochastic games}, International Journal of
  Game Theory \textbf{10} (1981), 53--66.

\bibitem{mignotte74}
M.~Mignotte, \emph{An inequality about factors of polynomials}, Mathematics of
  Computation \textbf{28} (1974), 1153--1157.

\bibitem{rao73}
S.~Rao, R.~Chandrasekaran, and K.P.K. Nair, \emph{Algorithms for discounted
  stochastic games}, Journal of Optimization Theory and Applications
  \textbf{11} (1973), 627--637.

\bibitem{renaultnotes2}
J.~Renault, \emph{{A tutorial on zero-sum stochastic games}}, ArXiv:1905.06577,
  2019.

\bibitem{shapley53}
L.S. Shapley, \emph{Stochastic games}, Proceedings of the National Academy of
  Sciences of the United States of America \textbf{39} (1953), 1095--1100.

\bibitem{SS50}
L.S. Shapley and R.N. Snow, \emph{Basic solutions of discrete games},
  Contributions to the Theory of Games, Vol.~I (H.W. Kuhn and A.W. Tucker,
  eds.), Annals of Mathematics Studies, vol.~24, Princeton University Press,
  Princeton, NJ, 1950, pp.~27--35.

\bibitem{SV10}
E.~Solan and N.~Vieille, \emph{Computing uniformly optimal strategies in
  two-player stochastic games}, Economic Theory \textbf{42} (2010), 237--253.

\bibitem{sorin02}
S.~Sorin, \emph{A {F}irst {C}ourse on {Z}ero-{S}um {R}epeated {G}ames},
  vol.~37, Springer Science \& Business Media, 2002.

\end{thebibliography}

\end{document}